\documentclass[preprint]{elsarticle}

\usepackage{amssymb,amsmath,amsthm, color}
\usepackage{inputenc}
\usepackage[margin=3cm]{geometry}
\usepackage{url}
\usepackage{amsmath}
\usepackage{pst-node}
\usepackage{tikz-cd} 

\newtheorem{theorem}{Theorem}[section]
\newtheorem{lemma}[theorem]{Lemma}
\newtheorem{define}[theorem]{Definition}
\newtheorem{cor}[theorem]{Corollary}
\newtheorem{prop}[theorem]{Proposition}

\newtheorem{example}[theorem]{Example}

\newcommand{\G}{\mathcal G}
\newcommand{\N}{\mathcal N}
\newcommand{\n}{\mathfrak n}
\newcommand{\D}{\mathfrak D}
\newcommand{\T}{\mathcal T}
\newcommand{\F}{\mathbb F}
\newcommand{\R}{\mathcal R}
\newcommand{\ppp}{\mathfrak p}
\newcommand{\pa}{\ppp^{\alpha}}
\newcommand{\pb}{\ppp^{\beta}}
\newcommand{\cyc}{\mathrm{Cyc}}
\newcommand{\m}{\mathfrak m}

\newcommand{\Z}{\mathbb Z}

\newcommand{\ord}[2]{\mathrm{ord}(#1,#2)}
\newcommand{\ou}{\mathcal{O}}
\newcommand{\rad}{\mathrm{rad}}
\newcommand{\rtp}{\widehat{\otimes}}
\newcommand{\amap}[1]{\Gamma_{a,#1}}
\newcommand{\ideal}[1]{\mathfrak{#1}}
\newcommand{\an}[1]{\langle #1\rangle}

\newcommand{\doublespace}

\begin{document}

\begin{frontmatter}

\title{Dynamics of the $a$-map over residually finite Dedekind Domains and applications}

\author[UNICAMP]{Claudio Qureshi}
\ead{cqureshi@gmail.com}
\author[USP]{Lucas Reis}
\ead{lucasreismat@gmail.com}
\address[UNICAMP]{Universidade Estadual de Campinas, Instituto de Matem{\'a}tica, Estat{\'\i}stica e Computa\c{c}{\~a}o Cient{\'\i}fica, Campinas, SP 13083-859, Brazil.}
\address[USP]{Universidade de S\~{a}o Paulo, Instituto de Ci\^{e}ncias Matem\'{a}ticas e de Computa\c{c}\~{a}o, S\~{a}o
Carlos, SP 13560-970, Brazil.}

\journal{Elsevier}
\begin{abstract}
Let $\D$ be a residually finite Dedekind domain, $a\in \D$ be a nonzero element and $\n$ be a nonzero ideal of $\D$. In this paper we describe the dynamics of the map $x\mapsto ax$ over the quotient ring $\D/\n$. We further present some applications of our main result.
\end{abstract}

\end{frontmatter}

\section{Introduction}


Finite dynamical systems associated with special types of functions have been extensively studied in the literature. For instance, iterations of quadratic polynomials over finite fields (motivated in part by some cryptographic applications such as the Pollard-rho factorization algorithm) were studied in \cite{PMMY01,Rogers96,VS04}. Dynamic of Chebyshev polynomials of prime degree and its relation with decomposition of primes in certain towers of number fields were studied by A. Gassert in \cite{Gassert14} and \cite{Gassert14b}. Dynamic of Chebyshev polynomials of arbitrary degree was studied in \cite{QP18}. Dynamic of special types of linearized polynomials over finite fields was described in \cite{PR18}. Dynamic of rational maps over finite fields such as R{\'e}dei functions \cite{QP15} and maps of the form $x \mapsto k(x+x^{-1})$ \cite{Ugolini13, Ugolini14} have also been considered. The dynamic of certain maps associated with endomorphism of ordinary elliptic curves over finite fields was dealt in \citep{Ugolini18}. A survey on iteration of functions over finite fields is given in \cite{MPQ19}. \\

In this paper, we consider a residually finite Dedekind Domain, which we denote by $\D$, a nonzero element $a\in \D$, a nonzero ideal $\n \unlhd \D$ and study the dynamic of the $a$-map $\amap{\n}$ given by: 
\begin{equation}\label{def:GammaMap}
\begin{array}{rccl}
  \amap{\n} : & \D/\n &\longrightarrow&\D/\n \\
  &x &\longmapsto& \Psi_{\n}(a)\cdot x.
\end{array}
\end{equation}
where $\Psi_{\n}:\D\to \D/\n$ is the canonical epimorphism (i.e. $\Psi_{\n}(x)=x+\n$). We bring a unified frame to the study of several dynamical systems, some of them mentioned above, via the dynamic of the maps $\Gamma_{a,\n}$. In general, the dynamic of a map $f$ over a finite set $X$ can be described through its associated functional graph $\mathcal{G}(f/X)$ whose vertices are the elements of $X$ and (directed) edges of the form $(x,f(x))$ for $x\in X$. In dynamical systems, two maps $f:X \rightarrow X$ and $g: Y \rightarrow Y$ are called {\it conjugates} when there is a bijection $h:X\rightarrow Y$ such that $h \circ f = g \circ h$. In this case, $h$ establishes an isomorphism between the functional graphs $\mathcal{G}(f/X)$ and $\mathcal{G}(g/Y)$. Our main result is a complete description of the functional graph $\mathcal{G}(\amap{\n})$ up to isomorphism (Theorem \ref{th:main}). This result naturally extends the structural theorem for the functional graphs associated with the $a$-map over cyclic groups and with R{\'e}dei functions over finite fields given in \citep{QP15}, see also \cite[Proposition 2.1]{QPM17}. Other corollary of our main result is a complete description of the functional graphs associated with linearized polynomials over finite fields, extending results given in \cite{PR18}. These polynomials have many interesting properties and appear in diverse areas such as network coding theory \cite{WAS13} and finite projective geometries; see for example Chapter 3.4 of \cite{LN97} and the notes at the end of this chapter for more properties and applications of these polynomials. The hanging trees attached to some periodic points of Chebyshev polynomials and maps induced by endomorphism of elliptic curves considered in \citep{Ugolini18}, both over finite fields, can also be explained from our main result.\\

This paper is organized as follows. In Section \ref{sect:preliminaries} we cover some preliminaries results and fix some notation to be used throughout this paper. In Section \ref{sect:main} we prove our main result (Theorem \ref{th:main}) and provide a concrete example. In Section \ref{sect:applications} we show several specializations of our main result; in particular we completely describe the dynamic of linearized polynomials over finite fields.

\section{Preliminaries}\label{sect:preliminaries}
In this section, we provide background material that is used along the way and some preliminary results heading to the proof of our main result.

\subsection{On residually finite Dedekind Domains}\label{subsect:OnRFDD}

Let $\D$ be a residually finite Dedekind Domain (i.e. $\D$ is an integral domain in which every nonzero proper ideal $\n$ factors into a product of prime ideals and the residue class ring $\D/\n$ is finite). For ideals $\n$ and $\m$ of $\D$, we denote $\m | \n$ if there is an ideal $\m' \unlhd \D$ such that $\n= \m \m'$. In a Dedekind domain, we have that $\m | \n$ if and only if $\n \subseteq \m$ and consequently $\gcd(\n,\m)=\n+\m$. When $\gcd(\n,\m)$ is principal, say $\gcd(\n,\m)=f \D$, we abuse of notation and write $\gcd(\n,\m)=f$. For example, $\gcd(\n,\m)=1$ means $\n+\m=\D$ and we say that $\n$ and $\m$ are relatively prime ideals. In this case we have $\n \cap \m = \n\m$. The radical of an ideal $\n$ is defined as $\mathrm{rad}(\n)=\{ d \in \D: d^i \in \n \mbox{ for some }i\in \Z^{+}  \}$. In a Dedekind domain $\rad(\n)$ is the product of the distinct prime ideals factors of $\n$. If $a \in \D$ is a nonzero element and $\n \unlhd \D$ is a nonzero ideal we have a unique decomposition $\n = \n_0 \n_1 $ where $\an{a}\subseteq \mathrm{rad}(\ideal{\n_0})$ and $\gcd(\n_1, \an{a}) = 1$; we refer to this decomposition as the $a$-decomposition of the ideal $\n$. \\

The norm, Euler Phi function and multiplicative order for Dedekind domains are defined as follows.

%

\begin{define}
Let $\n$ be any nonzero ideal of $\D$. 
\begin{enumerate}[(i)]
\item The \emph{norm} $\N_{\D}(\n)$ of $\n$ is the cardinality of the residual class ring $\D/\n$.
\item If $\n$ is a proper ideal of $\D$, the Euler Phi function $\varphi_{\D}(\n)$ of $\n$ is the cardinality of the group of units $U(\D/\n)$ of $\D/\n$. If $\n=\D$, $\varphi_{\D}(\n):=1$. 
\item For any element $a\in \D$ such that $\an{a}+\n=\D$, let $\ord{a}{\n}$ be the least positive integer $i$ such that $a^i-1\in \n$, or equivalently, $a^i\equiv 1\pmod {\n}$.
\end{enumerate}
\end{define}

It is well known that the norm $\N_{\D}$ is completely multiplicative, the Euler Phi function $\varphi_{\D}$ is multiplicative and, for any nonzero ideal $\n$ of $\D$:
$$\varphi_{\D}(\n)=\N_{\D}(\n)\cdot \prod_{\ppp|\n}\left(1-\frac{1}{\N_{\D}(\ppp)}\right),$$
where the product above is over all the distinct prime ideals dividing $\n$; see for example \cite[Chapter 1]{Narkiewicz04}. In particular, if $\ppp$ is any nonzero prime ideal of $\D$ and $i$ is a positive integer, we have that
\begin{equation}\label{eq:norm-phi}\varphi(\ppp^i)=\N_{\D}(\ppp^i)-\N_{\D}(\ppp^{i-1}).\end{equation}

We observe that, from definition, $\an{a}+\n=\D$ if and only if $\Psi_{\n}(a)\in U(\D/\n)$ and, in this case, $\ord{a}{\n}$ is the multiplicative order of $\Psi_{\n}(a)$. Then, by Lagrange theorem, $\ord{a}{\n} \mid \varphi_{\D}(\n)$. The following result provides information on the existence and number of solutions of linear congruences in Dedekind domains.



\begin{lemma}[{\cite[Theorem 2.3]{M14}}]\label{lem:linearcong}
Let $\D$ be a residually finite Dedekind Domain and $\n$ an ideal of $\D$. For $a, b\in \D$, the linear congruence 
$$ax\equiv b\pmod{\n},$$
is solvable if and only if $b\in \an{a}+\n$. Furthermore, if the congruence is solvable, then it has exactly $N(\an{a}+\n)$ incongruent solutions modulo $\n$.
\end{lemma}

Let $\ppp$ be any nonzero prime ideal of $\D$ and $\alpha \geq 1$. For each nonzero element $b\in \D$, we denote by $\nu_{\ppp}(b)$ the exponent of $\ppp$ in the factorial decomposition of $\an{b}$ into product of prime ideals. Since the ideals of $\D/\n$ are exactly the ideals of the form $\Psi_{\n}(\ideal{a})$ with $\ideal{a}\mid \n$, the quotient ring $\D/\pa$ is a finite local ring. The following result has easy verification.


\begin{lemma}\label{lem:aux-1}
Let $\alpha\ge 1$ and $\m$ be the maximal ideal of $\D/\pa$. The following statements hold:
\begin{enumerate}[(i)]
\item $\m$ is the homomorphic image of $\ppp$ by $\Psi_{\pa}$ and, in particular, $\m$ is principal;
\item for any nonzero element $b\in \D$ such that $\Psi_{\pa}(b)\ne 0$, $\nu_{\ppp}(b)$ is the only nonnegative integer $i$ such that $\Psi_{\pa}(b)\in \m^i \setminus \m^{i+1}$ (with the convention $\m^0=\D/\pa$);
\item for each $0\le i\le \alpha$, $|\m^i|=\N_{\D}(\ppp^{\alpha-i})$.
\end{enumerate}
\end{lemma}


From the previous lemma, we obtain the following result.

\begin{lemma}\label{lem:phi-1}
Let $\D$ be a residually finite Dedekind Domain and $\n\in \D$ a nonzero ideal. For each ideal $\m$ dividing $\n$, there exists $\varphi_{\D}(\n/\m)$ incongruent elements $b\in \D$ modulo $\n$ such that $\gcd(\an{b}, \n)=\m$.
\end{lemma}
\begin{proof}
If $\m=\n$, the result is trivial so we assume that $\m\ne \n$. Let $k({\m}, \n)$ be the number of incongruent elements $b\in \D$ modulo $\n$ such that $\gcd(\an{b}, \n)=\m$. From the Chinese remainder theorem, $k({\m}_0, \n_0)\cdot k(\m_1, \n_1)=k(\m_0\m_1, \n_0\n_1)$ whenever $\m_0+\m_1=\n_0+\n_1=\D$. So we only need to consider the case that $\n=\ppp^{\alpha}$ for some $\alpha\ge 1$ and some (nonzero) prime ideal $\ppp$ of $\D$ and $\m=\ppp^i$ with $0\le i<\alpha$. However, in this case, we observe that a nonzero element $b\in \D$ is such that $\gcd(\an{b}, \ppp^{\alpha})=\ppp^i$ with $0\le i<\alpha$, if and only if $\nu_{\ppp}(b)=i$, i.e., $\Psi_{\ppp^{\alpha}}(b)\in \m^i\setminus \m^{i+1}$. Therefore, $k(\ppp^i)=|\m^i\setminus \m^{i+1}|=|\m^i|-|\m^{i+1}|$ and so, from Lemma~\ref{lem:aux-1} and Eq.\eqref{eq:norm-phi}, we have that
$$k(\ppp^i)=|\m^i|-|\m^{i+1}|=\N_{\D}(\ppp^{\alpha-i})-\N_{\D}(\ppp^{\alpha-i-1})=\varphi_{\D}(\ppp^{\alpha-i}).$$
\end{proof}

\subsection{Operations on functional graphs, elementary trees and $\nu$-series}

Most definitions and notations introduced here are taken from \cite{QP15}. We denote
by $\bigoplus_{i=1}^{k}G_{i}$ the disjoint union of the graphs
$G_1,\ldots,G_k$ and $k\times G = \bigoplus_{i=1}^{k}G$ for $k\in \Z^{+}$. If $m\in \Z^+$ and $T$ is a rooted tree, we denote by
$\mathrm{Cyc}(m,T)$ a graph with a unique directed cycle of length $m$, where
every node in this cycle is the root of a tree isomorphic to $T$. The tree $T$ with a unique node is denoted by $\bullet$. An extended rooted tree is a graph of the form $\mathrm{Cyc}(1,T)$ for some rooted tree $T$, in this case we write $\{T\} = \mathrm{Cyc}(1,T)$. A {\it forest} is a disjoint union of rooted trees and an {\it extended forest} is a disjoint union of rooted trees and extended rooted trees.
If  $G= \bigoplus_{i=1}^{k}T_{i}$ is a forest where  $T_1,\ldots, T_k$ are rooted trees, we denote by
$\langle G \rangle$ a rooted tree verifying that its root has exactly $k$ predecessors $v_1,\ldots,v_k$ where $v_i$ is the
root of a tree isomorphic to $T_i$ for $i=1,\ldots,k$.  \\

To each non-increasing finite sequence of positive integers $V=(\nu_1,\nu_2,\ldots, \nu_d)$ (i.e. $\nu_1\geq \nu_2 \geq \cdots, \nu_d \geq 1$) we can associate a rooted tree $T_V$ defined recursively as follows:

\begin{equation}\label{TreeAssociatedEq}
\left\{ \begin{array}{l}
           T_{V}^{0}= \bullet,   \\
          G_{V}^{k}= \nu_k \times T^{k-1} \oplus
                 \bigoplus_{i=1}^{k-1}(\nu_{i}-\nu_{i+1})\times T^{i-1}
                  \textrm{ and } T_V^{k}=\langle G_{V}^{k} \rangle \textrm { for } 1\leq k <D, \\
    G_{V} = \langle (\nu_d-1) \times T^{D-1} \oplus
                 \bigoplus_{i=1}^{D-1}(\nu_{i}-\nu_{i+1})\times T^{i-1}
                 \rangle \textrm{ and } T_V= \langle G_{V} \rangle.
         \end{array}
\right.\end{equation}

Trees associated with non-increasing sequences as above are called {\it elementary trees}; see Figure \ref{fig:elementary_tree}. We note that, by definition, if $V=(\nu_1,\ldots, \nu_d)$ and $W=(\nu_1,\ldots,\nu_d, 1, 1, \ldots, 1)$ then $T_V= T_W$.
\begin{figure}[ht]
\centering
\includegraphics[width=0.9\textwidth]{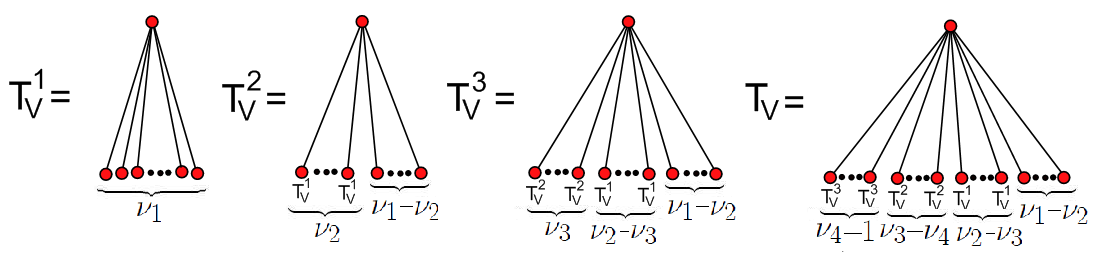}
\vspace*{-8pt}
\caption{This figure (taken from \cite{QP15}) illustrates the inductive
definition of $T_V$ for $V=(\nu_1,\nu_2,\nu_3,\nu_4)$. A node $v$ labelled by a rooted tree $T$
indicates that $v$ is the root of a tree isomorphic to $T$.}
\label{fig:elementary_tree}
\end{figure}
Next we extend the definition of $\nu$-series introduced in \cite{QP15} to Dedekind domains.

\begin{define}
Let $\D$ be a Dedekind domain, $a\in \D$ be a nonzero element and $\n_0 \unlhd \D$ be a nonzero ideal such that $\an{a} \subseteq \mathrm{rad}(\n_0)$. The $\nu$-series associated with $\n_0$ and $a$ is the sequence of positive integers $\n_0(a):=(N(\ideal{a}_1),\ldots, N(\ideal{a}_d))$ where $N(\ideal{a})$ denotes the norm of the ideal $\ideal{a}$ and the ideals $\ideal{a}_1,\ldots, \ideal{a}_d$ are given as follows:
$$\ideal{a}_1=\gcd(\n_0,\an{a}), \quad \ideal{a}_{i+1}=
  \gcd\left(\frac{\n_0}{\ideal{a}_1\cdots \ideal{a}_i}, \an{a}\right), \mbox{ for } 1\leq i \leq d,$$
where $d$ is the least positive integer such that $\ideal{a}_1\cdots \ideal{a}_d =\n_0$. When $\n_0=\langle b \rangle$ is a principal ideal of $\D$ we denote $b(a):=\n_0(a)$. 
\end{define}

The next result is a direct consequence of the fact that $\gcd(\ideal{q}\ideal{q}', \ideal{n})= \gcd(\ideal{q}, \ideal{n})\cdot \gcd(\ideal{q}', \ideal{n})$ if $\gcd(\ideal{q},\ideal{q}')=1$ and the definition of $\nu$-series.

\begin{lemma}\label{lem:vseriesproduct}
Let $\D$ be a Dedekind domain, $a\in \D$ be a nonzero element and $\ideal{q},\ideal{q}'$ be nonzero ideals such that $\an{a} \subseteq \mathrm{rad}(\ideal{q})$, $\an{a} \subseteq \mathrm{rad}(\ideal{q}')$ and $\gcd(\ideal{q}, \ideal{q}')=1$. Suppose that $\ideal{q}(a):=(N(\ideal{a}_1),\ldots, N(\ideal{a}_d))$ and $\ideal{q}'(a):=(N(\ideal{a}'_1),\ldots, N(\ideal{a}'_{d'}))$ with $d'\leq d$. Then, $\ideal{q}\ideal{q}'(a):=(N(\ideal{b}_1),\ldots, N(\ideal{b}_d))$ with 
$$\ideal{b}_i = \left\{ \begin{array}{ll} \ideal{a}_i \ideal{a}'_i & \textrm{for }1\leq i \leq d';\\
\ideal{a}_i & \textrm{for } d'<i\leq d .
\end{array}  \right.$$
\end{lemma}

The tree $T_{\n_0(a)}$ associated with the $\nu$-series $\n_0(a)$ plays an important role in the description of the functional graph $\mathcal{G}(\Gamma_{a,\n})$.

\subsection{Multiplicativity of $\nu$-series trees}

Here we prove a multiplicative property of elementary trees with respect to the tensor product. We recall that if $G_1$ and $G_2$ are directed graphs with vertex sets $V_{G_1}$ and $V_{G_2}$, respectively, their tensor product $G_1 \otimes G_2$ is a directed graph with vertex set $V_{G}= V_{G_1} \times V_{G_2}$ and $(v_1,v_2)\to (w_1,w_2)$ is an edge in $G_1\otimes G_2$ if and only if $v_i\to w_i$ is an edge in $G_i$, for $i=1,2$. Note that the tensor product is commutative (i.e., $G_1\otimes G_2$ and $G_2\otimes G_1$ are isomorphic) and distributive with respect to the disjoint sum of graphs $\oplus$. For two non-increasing sequences of positive integers $U=(u_1, \ldots, u_d)$ and $V=(v_1, \ldots, v_d)$ we define their product $UV = VU := (u_1v_1, \ldots, u_dv_d)$.\\

We note that if $T$ and $T'$ are rooted trees with $N$ and $N'$ nodes, their tensor product $T \otimes T'$ is a forest with exactly $N+N'-1$ rooted trees, where the roots are exactly the vertices of the form $(r_T,t')$ with $t' \in T'$ and $(t,r_{T'})$ with $t \in T$ (here $r_T$ and $r_{T'}$ denote the roots of $T$ and $T'$, respectively). It is convenient to introduce a new operation.

\begin{define}
Let $T$ and $T'$ be two rooted trees with roots $r_T$ and $r_{T'}$, respectively. Set $T_1 = T$ or $\{T\}$, and $T_2 = T'$ or $\{T'\}$. the {\it restricted tensor product} of $T_1$ and $T_2$ (denoted by $T_1 \rtp T_2$) is the connected component of $T_1 \otimes T_2$ containing the node $(r_T, r_T')$. 
\end{define} 

Note that $T \otimes T'$ is a rooted tree with root $(r_T,r_{T'})$ (it is the hanging tree of $(r_T, r_{T'})$ in $T\otimes T'$). If we denote by $d(x,y)$ the length of the smaller directed path from $x$ to $y$ (if there is any), we have:
\begin{itemize}
\item $V_{T \rtp T'} = \{(x,x') \in V_{T}\times V_{T'}: d(x,r_T)=d(x,r_{T'})   \}$, 
\item $V_{T \rtp \{T'\}} = \{(x,x') \in V_{T}\times V_{T'}: d(x,r_{T})\geq d(x',r_{T'})  \}$,
\item $V_{\{T\} \rtp T'} = \{(x,x') \in V_{T}\times V_{T'}: d(x,r_{T})\leq d(x',r_{T'})  \}$.
\end{itemize}

The next proposition gives some useful properties of the restricted tensor product whose proofs are straightforward.

\begin{prop}\label{prop:RestrictedTensorProduct}
Let $G$ and $G'$ be two forests. Let $T=\langle G \rangle$ and $T'=\langle G' \rangle$. The next properties hold:
\begin{enumerate}[(i)]
\item $T \rtp T'  = \langle G \rtp G' \rangle$;
\item $T \rtp \{T'\} = \langle G \rtp G' \oplus G \rtp \{T'\} \rangle$ and $\{T\} \rtp T' = \langle G \rtp G' \oplus \{T\} \rtp G' \rangle$;
\item $\{T\} \otimes \{T'\} =  \langle G \rtp G' \oplus G \rtp \{T'\} \oplus \{T\} \rtp G' \rangle$.
\end{enumerate}
\end{prop}

The next lemma establishes a relation between the partial trees associate with $\nu$-series.

\begin{lemma}\label{lem:tree-1}
Let $V=(v_1, \ldots, v_d)$ and $U=(u_1, \ldots, u_d)$ be non-increasing sequences of positive integers. Then the following hold:

\begin{enumerate}[(i)]
\item $T_V^i\rtp T_U^j=T_{UV}^{\min\{i, j\}}$ for any $0\le i, j\le d$;
\item $T_V^i\rtp \{T_U\}=T_{UV}^i=T_U^i\rtp \{T_V\}$ for any $0\le i\le d$.
\end{enumerate}

\end{lemma}

\begin{proof}

We proceed by induction on $s=\min\{i, j\}$ to prove part (i). If $s=0$, the result is trivial. Suppose that the result holds for any $l\le s$ with $s\ge 0$ and let $0\le i, j\le d$ be such that $\min\{i, j\}=s+1$. Hence, $s+1\ge 1$. Without loss of generality, suppose that $i=s+1 \leq j$. Using item (i) of Proposition \ref{prop:RestrictedTensorProduct} we have that

$$T_V^i\rtp T_U^j=\left\langle \left(v_{s+1}\times T_V^{s}\oplus \bigoplus_{k=1}^{s}(v_k-v_{k+1})\times T_V^{k-1}\right)\rtp  \left(u_j\times T_U^{j-1}\oplus \bigoplus_{l=1}^{j-1}(u_l-u_{l+1})\times T_U^{l-1}\right)\right\rangle.$$

From induction hypothesis, $T_V^a\rtp T_U^b=T_{UV}^{\min\{a, b\}}$ whenever $0\le \min\{a, b\}\le s$. From this fact, we may infer that
$$T_V^i\rtp T_U^j=\left\langle w_{s+1}\times T_{UV}^s\oplus \bigoplus_{k=1}^sw_k\times T_{UV}^{k-1}\right\rangle,$$
where the numbers $w_k$ are given as follows. For each $1\le k\le s$,
$$w_k=(v_{k}-v_{k+1})\left(u_j+\sum_{k\le l\le j-1}(u_l-u_{l+1})\right)+(u_k-u_{k+1})\left(v_{s+1}+\sum_{k<l\le s}(v_l-v_{l+1})\right)=$$
$$(v_k-v_{k+1})u_k+(u_k-u_{k+1})v_{k+1}=u_kv_k-u_{k+1}v_{k+1}.$$
In addition, 
$$w_{s+1}=v_{s+1}\left(u_j+\sum_{s+1\le l\le j-1}(u_l-u_{l+1})\right)=u_{s+1}v_{s+1}.$$

Now we prove part (ii) by induction on $i$. The case $i=0$ is straightforward. Let $i: 1\leq i \leq d$ and suppose that the result holds for any $k: 0\leq k <i$. Using part (ii) of Proposition \ref{prop:RestrictedTensorProduct} we have that 
$$T_V^i\rtp \{T_U\}=\left\langle G_{V}^{i} \rtp \left( G_{U} \oplus \{T_{U}\} \right) \right\rangle.  $$
Then, we write $G_{V}^{i} \rtp \left( G_{U} \oplus \{T_{U}\} \right)$ as a disjoint sum of products of the form $T_V^k\rtp T_U^l$ and $T_V^k\rtp \{T_U\}$ (this last case, $k\le i-1$), which are easily computed using item (i) and the induction hypothesis. Reordering the terms as in the proof of item (i), we easily obtain the desired identity.  


\end{proof}

\begin{prop}\label{prop:tree-1}
For $V=(v_1, \ldots, v_d)$ and $U=(u_1, \ldots, u_d)$ with $\{v_i\}_{1\le i\le d}$ and $\{u_i\}_{1\le i\le d}$ non decreasing sequences, set $UV=(u_1v_1, \ldots, u_dv_d)$. Then the following holds:
$$\{T_U\}\otimes \{T_V\}=\{T_{UV}\}.$$
\end{prop}

\begin{proof}
By item (iii) of Proposition \ref{prop:RestrictedTensorProduct} we have 
\begin{equation}\label{eq:TuTv}
\{T_U\}\rtp \{T_V\}=\left\langle (G_{U} \rtp G_{V}) \oplus (G_{U} \rtp \{T_V\}) \oplus (\{T_{U}\} \rtp G_{V}) \right\rangle.
\end{equation}
Using item (i) of Lemma \ref{lem:tree-1} we compute
\begin{align}
G_{U} \rtp G_{v} &= \left( (u_d-1)\times T_U^{d-1}\oplus \bigoplus_{l=1}^{d-1}(u_l-u_{l+1})\times T_U^{l-1}  \right) \rtp \left( (v_d-1)\times T_V^{d-1}\oplus \bigoplus_{l=1}^{d-1}(v_l-v_{l+1})\times T_V^{l-1}  \right) \nonumber \\
&= (u_d-1)(v_d-1)\times T_{UV}^{d-1} \oplus \bigoplus_{k=1}^{d-1}\left[ (u_{k}-u_{k+1})(v_{k}-1) + (v_{k}-v_{k+1})(u_{k+1}-1) \right]\times T_{UV}^{k-1}. \label{eq:TuTvI}
\end{align}
Using item (ii) of Lemma \ref{lem:tree-1} we compute
\begin{align}
G_{U} \rtp \{T_V\} &=  (u_d-1)\times T_{UV}^{d-1} \oplus \bigoplus_{k=1}^{d-1} (u_{k}-u_{k+1})\times T_{UV}^{k-1}; \label{eq:TuTvII} \\
\{T_{U}\} \rtp G_{V} &=  (v_d-1)\times T_{UV}^{d-1} \oplus \bigoplus_{k=1}^{d-1} (v_{k}-v_{k+1})\times T_{UV}^{k-1} \label{eq:TuTvIII}.
\end{align}
Substituting Equations (\ref{eq:TuTvI}), (\ref{eq:TuTvII}) and (\ref{eq:TuTvIII}) into Equation (\ref{eq:TuTv}) we obtain the desired result.

\end{proof}

\begin{cor}\label{cor:tree-1}
Let $\D$ be a residually finite Dedekind domain, $a\in \D$ be a nonzero element and $\ideal{q},\ideal{q}'$ be nonzero ideals such that $\an{a} \subseteq \mathrm{rad}(\ideal{q})$, $\an{a} \subseteq \mathrm{rad}(\ideal{q}')$ and $\gcd(\ideal{q}, \ideal{q}')=1$. Then, 
$$  \{T_{\ideal{q}(a)}\} \otimes \{T_{\ideal{q}'(a)}\} = \{ T_{\ideal{q}\ideal{q}'(a)} \}.  $$
\end{cor}

\begin{proof}
Let $\ideal{q}(a):=(N(\ideal{a}_1),\ldots, N(\ideal{a}_d))$ and $\ideal{q}'(a):=(N(\ideal{a}'_1),\ldots, N(\ideal{a}'_{d'}))$. Without loss of generality we can suppose $d'\leq d$. We consider the $d$-terms sequences $U=\ideal{q}(a)$ and $V=(N(\ideal{a}'_1),\ldots, N(\ideal{a}'_{d'}), 1,  \cdots, 1)$. By Lemma \ref{lem:vseriesproduct}, $UV=\ideal{q}\ideal{q}'(a)$. We conclude noting that $T_{V}=T_{\ideal{q'}(a)}$ and using Proposition \ref{prop:tree-1}. 
\end{proof}

\section{Proof of the main result}\label{sect:main}

We consider here a residually finite Dedekind domain $\D$, a nonzero element $a\in \D$, a nonzero ideal $\n \unlhd \D$ and the $a$-map $\amap{\n}$ defined as in Equation (\ref{def:GammaMap}). Let $\n = \n_0\n_1$ be the $a$-decomposition of the ideal $\n$ (i.e. $\an{a}\subseteq \n_0$ and $\gcd(\an{a},\n_1)=1$). If $y \in \D/\n$ is a periodic point of $\amap{\n}$ we denote by $c_{a, \n}(y)$ its period, that is, the least positive integer $i$ such that $\Gamma_{a, \n}^{(i)}(y)=y$, where $f^{(n)}$ denotes the composition of $f$ with itself $n$ times.

\subsection{The case $\gcd(\an{a},\n)=1$}

In this case there is $a' \in \D$ such that $aa'\equiv 1 \pmod{\n}$ and the map $\amap{\n}$ is invertible (with inverse $\Gamma_{a',\n}$). Then, in this case every point is periodic and the graph $\mathcal{G}(\amap{\n})$ is a disjoint union of cycles. The next result brings an explicit description of the graph $\mathcal{G}(\amap{\n})$.

\begin{prop}\label{prop:cyclic-case}
Let $\D$ be a residually finite Dedekind Domain, $\n\in \D$ be a nonzero ideal and $a\in \D$ be an element such that $\an{a}+\n=\D$. For each $y=b+\n \in \D/\n$, the following statements hold.

\begin{enumerate}[(i)]
\item If $i$ is a positive integer, $(a^i-1) b\in \n$ if and only if $\Psi_{\n}(a)^i y = y$. In particular, $c_{a, \n}(y)=\ord{a}{\n'}$ where $\n'=\frac{\n}{\gcd(\an{b}, \n)}$.
\item If $b_0 \in \D$ is such that $\Gamma_{a,\n}(y)=b_0+\n$, then $\gcd(\an{b_0}, \n)=\gcd(\an{b}, \n)$.
\end{enumerate}
In particular, the cycle decomposition of the map $\Gamma_{a, \n}$ over $\D/\n$ is given as follows
$$\bigoplus_{\m | \n}\frac{\varphi_{\D}(\m)}{\ord{a}{\m}}\times \cyc(\ord{a}{\m}, \bullet),$$
where the sum is over the distinct ideals $\m$ of $\D$ such that $\m$ divides $\n$.
\end{prop}

\begin{proof}
We split the proof into cases.

\begin{enumerate}[(i)]
\item We have the following chain of equivalences:  $(a^i-1)b\in\n \Leftrightarrow a^i b \equiv b \pmod{\n} \Leftrightarrow \Psi_{\n}(a^i b) = \Psi_{\n}(b) \Leftrightarrow \Psi_{\n}(a)^i y = y \Leftrightarrow  \Gamma_{a,\n}^{(i)}(y)=y$ and also $a^i b \equiv b \pmod{\n} \Leftrightarrow a^i \equiv 1 \pmod{\n'}$. 
\item Since $\Gamma_{a,\n}(y)=ab+\n$ we have $ab\equiv b_0 \pmod{\n}$ and consequently $\gcd(\langle ab \rangle,\n)= \gcd(\langle b_0 \rangle, \n)$. On the other hand $\gcd(\langle ab \rangle,\n) = \gcd(\langle a \rangle \langle b \rangle,\n)= \gcd(\langle b \rangle, \n)$ because $\gcd(\langle a \rangle, \n)=1$.
\end{enumerate}
For each ideal $\m$ dividing $\n$, let $C_{\m}$ be a complete set of incongruent elements $b\in \D$ modulo $\n$ such that $\gcd(\an{b}, \n)=\frac{\n}{\m}$. In particular, $\D/\n$ equals the disjoint union of $\Psi_{\n}(C_{\m})$ with $\m \mid \n$. By item (ii), each set $\Psi_{\n}(C_{\m})$ is $\amap{\n}$-invariant, i.e. $\mathcal{G}(\amap{\n}) = \bigoplus_{\m \mid \n} \mathcal{G}(\amap{\n}/\Psi_{\n}(C_{\m}))$. From Lemma~\ref{lem:phi-1}, $\#\Psi_{\n}(C_{\m})= \#C_{\m} = \varphi_{\D}(\m)$, and by item (i), for each $b\in C_m$, the element $\Psi_{\n}(b)$ belongs to a cycle of length $c_{a, \n}(\Psi_{\n}(b))=\ord{a}{\m}$. Therefore, the restriction of $\Gamma_{a,\n}$ to $\Psi_{\n}(C_{\m})$ splits into $\frac{\varphi_{\D}(\m)}{\ord{a}{ \m}}$ cycles, each of length $\ord{a}{\m}$.
\end{proof}


\subsection{The case $\n=\ideal{p}^{\alpha}$ with  $\ideal{p} \mid \an{a}$}

When $\nu_{\ideal{p}}(a) \geq \alpha$ the dynamics of the map $\amap{\pa}$ is trivial (everyone goes to zero), so we focus on the situation when $\an{a}=\ideal{p}^{\beta}\cdot \ideal{b},$ for some ideal $\ideal{b}$ such that $\gcd(\ideal{b}, \ideal{p})=1$ and $\beta=\nu_{\ideal{p}}(a)<\alpha$ (in particular $\alpha\ge 2$). 
\begin{lemma}\label{lem:trees}
Write $d=\left\lfloor\frac{\alpha}{\beta}\right\rfloor$, $e=\alpha-d\beta<\beta$ and let $\m$ be the homomorphic image of $\ppp$ by $\Psi_{\pa}$. Then the following hold.
\begin{enumerate}[(i)]
\item The element $0\in \D/\n$ has exactly $\N_{\D}(\pb)$ preimages by $\amap{\pa}$, one being the element $0$ itself. In addition, the set of the other preimages equals the union of the sets $C_1=\m^{\alpha-\beta}\setminus \m^{\alpha-e}$ and $C_2=\m^{\alpha-e}\setminus\{0\}$.
\item Let $b\in \D$ such that  $\Psi_{\pa}(b)=y$ is nonzero. Then $y$ has preimages by $\amap{\pa}$ if and only if $\nu_{\ppp}(b)\ge \beta$. In this case, $y$ has exactly $\N_{\D}(\pb)$ preimages by $\amap{\pa}$ and, for any preimage $z\in \D/\n$ of $y$ by $\amap{\pa}$ and any $b_0\in \D$ such that $\Psi_{\pa}(b_0)=z$, we have that $\nu_{\ppp}(b_0)=\nu_{\ppp}(b)-\beta$.
\end{enumerate}
\end{lemma}

\begin{proof}
We split the proof into cases.

\begin{enumerate}[(i)]
\item Clearly, $\amap{\pa}(0)=0$. We observe that the number of preimages of $0$ by $\amap{\pa}$ equals the number of incongruent solutions modulo $\n$ for the linear congruence
$$ax\equiv 0\pmod {\n}.$$
From Lemma~\ref{lem:linearcong}, this number equals $\N_{\D}(\pb)$. In addition, any nonzero element $y\in \D/\n$ with $\amap{\pa}(y)=0$ satisfies $\Psi_{\pa}(a)y=0$, i.e., $\Psi_{\pa}(a)y\in \m^{\alpha}$. From Lemma~\ref{lem:aux-1}, $\Psi_{\pa}(a)\in \m^{\beta}\setminus \m^{\beta+1}$ and, since $\m$ is principal, we have that $y\in  \m^{\alpha-\beta}\setminus\{0\}=C_1\cup C_2$.

\item We observe that $y$ has preimages by $\amap{\pa}$ if and only if the linear congruence $ax\equiv b\pmod {\pa}$ has solution. From Lemma~\ref{lem:linearcong}, the latter is equivalent to $b\in \an{a}+\pa$. Since $\nu_{\ppp}(a)=\beta<\alpha$, $\an{a}+\pa=\pb$ and so $b\in \an{a}+\pa$ if and only if $\nu_{\ppp}(b)\ge \beta$. If the latter occurs, Lemma~\ref{lem:linearcong} entails that the linear congruence $ax\equiv b\pmod {\pa}$ has $\N_{\D}(\pb)$ incongruent solutions modulo $\pa$, i.e., $y$ has $\N_{\D}(\pb)$ preimages by $\amap{\pa}$. In addition, if $z=\Psi_{\pa}(b_0)$ is any preimage of $y$ by $\amap{\pa}$, we have that $ab_0\equiv b\pmod{\pa}$. Since $y$ is nonzero, $\nu_{\ppp}(b)<\alpha$ and so $\nu_{\ppp}(ab_0)=\nu_{\ppp}(b)$. Therefore, $\nu_{\ppp}(b_0)=\nu_{\ppp}(b)-\nu_{\ppp}(a)=\nu_{\ppp}(b)-\beta$.
\end{enumerate}
\end{proof}

From Lemma~\ref{lem:trees}, we can describe the dynamics of the map $\amap{\pa}$.


\begin{prop}\label{prop:tree-primary}
Let $\ideal{q}=\ideal{p}^{\alpha}$ and $a$ be a nonzero element of $\D$ such that $\ideal{p}\mid \an{a}$. Then,
$$ \mathcal{G}(\amap{q}) = \{ T_{\ideal{q}(a)}  \}.$$
\end{prop}

\begin{proof}
Let $\nu_{\ideal{p}}(a)=\beta$. First note that if $\alpha \leq \beta$ then $\ideal{q}(a)=(\N(\ideal{q}))$ where $\N(\ideal{q})= \# \D/ \ideal{q}$. Thus $T_{\ideal{q}(a)}= \langle (\N(\ideal{q})-1)\times \bullet \rangle$, that is, the tree consisting of one root and $\N(\ideal{q})-1$ predecessors. Since in this case everyone goes to zero by $\amap{\ideal{q}}: \D/\ideal{q} \rightarrow \D/ \ideal{q}$, their functional graph consist of a loop at $0$ with $\N(\ideal{q})-1$ predecessors and then $\mathcal{G}(\amap{q}) = \{ T_{\ideal{q}(a)}  \}$. Now, we assume $\alpha > \beta$ and write $\alpha = d\beta +e $ with $0\leq e <\beta$ and $d\geq 1$. In this case $$\ideal{q}(a)=(\underbrace{\N_{\D}(\ppp^{\beta}), \ldots, \N_{\D}(\ppp^{\beta})}_{d\mathrm{\; times}},\N_{\D}(\ppp^{e})).$$ From item (ii) of Lemma~\ref{lem:trees}, any nonzero preimage $y\in \D/\pa$ of the element $0$ such that $y=\Psi_{\pa}(b)$ and $\nu_{\ppp}(b)=t$ is the root of a $\N_{\D}(\ppp^{\beta})$-ary complete tree of height $k=\left\lfloor\frac{t}{\beta}\right\rfloor$, and therefore isomorphic to $T_{\ideal{q}(a)}^{k}$. We observe that $\left\lfloor\frac{s}{\beta}\right\rfloor=d$ and $\left\lfloor\frac{s}{\beta}\right\rfloor=d-1$ according to $\alpha-\beta\le s< \alpha-e$ and $\alpha-e\le s\le \alpha$, respectively. In particular, from item (i) of Lemma~\ref{lem:trees} and item (ii) of Lemma~\ref{lem:aux-1}, $|\m^{\alpha-\beta}\setminus \m^{\alpha-e}|$ preimages of $0$ have a hanging tree isomorphic to $T_{\ideal{q}(a)}^{d-1}$ and the remaining $|\m^{\alpha-e}\setminus\{0\}|$ nonzero preimages of $0$ have a hanging tree isomorphic to $T_{\ideal{q}(a)}^{d}$. From Lemma~\ref{lem:aux-1}, $|\m^{\alpha-\beta}\setminus \m^{\alpha-e}|=\N_{\D}(\ppp^{\beta})-\N_{\D}(\ppp^{e})$ and $|\m^{\alpha-e}\setminus\{0\}|=\N_{\D}(\ppp^{e})-1$. In other words, we proved that the hanging tree of $0$ is isomorphic to $\langle (\N_{\D}(\ppp^{e})-1)\times T_{\ideal{q}(a)}^{d} \oplus (\N_{\D}(\ppp^{\beta})-\N_{\D}(\ppp^{e}))\times T_{\ideal{q}(a)}^{d-1} \rangle = T_{\ideal{q}(a)}$ (by the recursive definition of $T_{\ideal{q}(a)}$) and since $0$ is the only periodic point of $\amap{\ideal{q}}$ we conclude that $\mathcal{G}(\amap{q}) = \{ T_{\ideal{q}(a)}  \}$.
\end{proof}

\subsection{The general case}

Here we use the description obtained in the previous subsections to obtain the description of $\mathcal{G}(\amap{\n})$ for the general case. First we state some lemmas whose proofs are straightforward. In \cite{T05} it is considered the following product of maps of finite sets: for $f:X\to X$ and $g:Y\to Y$, their product is defined as the map $f\times g:X\times Y\to X\times Y$ given by $(x, y)\mapsto (f(x), g(y))$.

\begin{lemma}\label{lem:FDS}
For any maps of finite sets $f:X\to X$ and $g:Y\to Y$ we have the following graph isomorphism:
$$\G(f\times g/ X\times Y) \cong \G(f/X)\otimes \G(g/Y).$$
\end{lemma}

\begin{lemma}\label{lem:FDS-2}
Let $r\in \mathbb{Z}^{+}$ and $T$ be a rooted tree. The following isomorphism holds: $$\mathrm{Cyc}(r,\bullet)\otimes \{T\} \cong \mathrm{Cyc}(r,T). $$
\end{lemma}

An ideal $\ideal{q}$ of a Dedekind domain $\D$ is called {\it primary} if it is a power of a prime ideal (i.e. $\ideal{q}=\ideal{p}^s$ for some prime ideal $\ideal{p}\unlhd \D $ and $s \in \Z^+$). We recall that if $\n_0$ is any ideal of $\D$, there is a unique decomposition into primary ideals $\n_0= \ideal{q}_1 \cdots \ideal{q}_s$ (called the {\it primary decomposition} of $\n_0$).

\begin{theorem}\label{th:main}
Let $\D$ be a residually finite Dedekind domain, $a\in D$ be a nonzero element and $\n \unlhd \D$ be a nonzero ideal. Let $\n=\n_0\n_1$ be the $a$-decomposition of $\n$ (i.e. $\an{a}\subseteq \rad(\n_0)$ and $\gcd(\an{a},\n_1)=1$). Then, the following isomorphism holds:
$$ \mathcal{G}(\amap{\n}) = \bigoplus_{\m | \n_1} \frac{\varphi_{\D}(\m)}{\ord{a}{\m}} \times \cyc\left(\ord{a}{\m}, T_{\n_0(a)}\right);$$ where $\n_0(a)$ is the $\nu$-series associated with $\n_0$ and $a$.
\end{theorem}

\begin{proof}
Since $\gcd(\n_1,\n_0)=1$, by the Chinese remainder theorem the map $\eta: \D/\n \rightarrow \D/\n_1 \times \D/\n_0$ given by $\eta(d+\n)=(d+\n_1, d+\n_0)$ provides an isomorphism of $\D$-modules. Consequently, denoting by $\Gamma_{a,\n}^*=\Gamma_{a,\n_1} \times \Gamma_{a,\n_0}$ we have the following commutative diagram 
\[ \begin{tikzcd}
\D/\n \arrow{r}{\Gamma_{a, \n}} \arrow[swap]{d}{\eta} & \D/\n \arrow{d}{\eta} \\%
\D/\n_1\times \D/\n_0 \arrow{r}{\Gamma_{a,\n}^*}& \D/\n_1\times \D/\n_0.
\end{tikzcd}
\]

%

Then, $\eta$ induces a graph isomorphism $\G(\amap{\n}) \cong \G(\amap{\n_1} \times \amap{\n_0})$ and by Lemma \ref{lem:FDS} we obtain: 
\begin{equation}\label{eq:graph-1}
\G(\amap{\n})  \cong  \G(\amap{\n_1}) \otimes \G(\amap{\n_0}).
\end{equation}

Now we consider the primary decomposition of $\n_0= \ideal{q}_1\cdots \ideal{q}_s$. In a similar way, using the Chinese remainder theorem together with Lemma \ref{lem:FDS} we obtain a graph isomorphism $\G(\amap{\n_0}) \cong \bigotimes_{i=1}^{s}\G(\amap{\ideal{q}_i})$. By Proposition \ref{prop:tree-primary} and Corollary \ref{cor:tree-1} we have:
$$ \G(\amap{\n_0}) \cong \bigotimes_{i=1}^{s}\{T_{\ideal{q}_i(a)}\} \cong \{T_{\n_0(a)}\}.$$

Substituting $\G(\amap{\n_1})$ with the expression given in Proposition \ref{prop:cyclic-case} and $\G(\amap{\n_0})$ with $\{T_{\n_0(a)}\}$  into Equation (\ref{eq:graph-1}) and applying Lemma \ref{lem:FDS-2} we have:
\begin{align*}
\G(\amap{\n})  &\cong  \bigoplus_{\m | \n}\frac{\varphi_{\D}(\m)}{\ord{a}{\m}}\times \cyc(\ord{a}{\m}, \bullet) \otimes \{T_{\n_0}(a)\}\\ &\cong \bigoplus_{\m | \n}\frac{\varphi_{\D}(\m)}{\ord{a}{\m}}\times \cyc(\ord{a}{\m}, T_{\n_0}(a))
\end{align*}

\end{proof}

\begin{example}
We provide a single example for Theorem~\ref{th:main}. We consider $\D=\Z[\sqrt{-5}]$, $a=1+\sqrt{-5}$ and $\n=\an{6}$. We observe that $\an{a}=\ppp_1\ppp_2$ and $\n=\ppp_1^2\ppp_2\ppp_3$, where $\ppp_1=\an{2, 1+\sqrt{-5}}, \ppp_2=\an{3, 1+\sqrt{-5}}$ and $\ppp_3=\an{3, 2+\sqrt{-5}}$. 
In the notation of Theorem~\ref{th:main}, we have that $\n=\n_0\n_1$, where $\n_0=\ppp_1^2\ppp_2$ and $\n_1=\ppp_3$. We obtain $\ord{a}{\ppp_3}=\varphi_{\D}(\ppp_3)=2$. The $\nu$-series associated with $\an{a}$ and $\n_0$ equals $\n_0(a)=(\N_{\D}(\ppp_1\ppp_2),\N_{\D}( \ppp_1))=(6, 2)$. From Theorem~\ref{th:main}, we have that
$$\G(\Gamma_{1+\sqrt{-5}, \an{6}})=\cyc(1, T_{(6, 2)})\oplus\cyc(2, T_{(6, 2)}).$$
The graph $\G(\Gamma_{1+\sqrt{-5}, \an{6}})$ is explicitly shown in Figure~\ref{fig:func-graph}.
\end{example}

\begin{figure}[ht]
  \centering
    \includegraphics[width=0.6\linewidth]{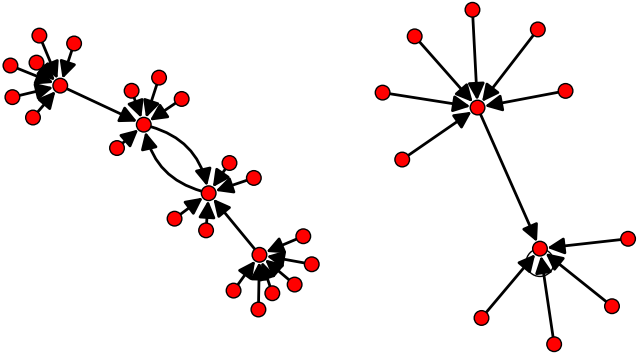}
  \caption{The functional graph of the map $x\mapsto (1+\sqrt{-5})\cdot x$ over $\frac{\Z[\sqrt{-5}]}{\an{6}}$.}
  \label{fig:func-graph}
\end{figure}

\section{Some applications}\label{sect:applications}

In this section we apply Theorem \ref{th:main} to some special families of maps over finite fields. First we show how to apply it to recover some known results about the dynamic of R{\'e}dei functions and Chebyshev polynomials. Then, we apply Theorem \ref{th:main} to describe generic trees in the functional graph of certain maps induced by endomorphism of ordinary elliptic curves. The description of these trees in terms of elementary trees associated with $\nu$-series is new. Finally, we use Theorem \ref{th:main} to obtain a complete description of the dynamic of linearized polynomials over finite fields (this description was only known for special cases). We fix some notation. Let $\F_q$ denote the finite field with $q$ elements and $\chi_q$ be the quadratic character of $\F_q$, that is, $\chi_q(a)=1$ if $a$ is a nonzero square in $\F_q$, $\chi_q(a)=-1$ if $a$ is a nonsquare in $\F_q$ and $\chi_q(0)=0$. For $n\geq 1$, $\F_{q^n}$ denotes the $n$-degree extension of the finite field $\F_q$.

\subsection{R{\'e}dei functions over finite fields}

The classical definition of R{\'e}dei function considers, for each positive integer $n$, 
the binomial expansion $(x+\sqrt{y})^{n} = N(x,y) + D(x,y) \sqrt{y}$ in two indeterminates $x$ and $y$.
Then, the R{\'e}dei function $R_n(x,a)$ of degree $n$ and parameter $a\in \F_q^{*}$ is the map given by $R_n(x,a)=\frac{N(x,a)}{D(x,a)}$ defined over the projective line $\mathbb{P}^{1}(\F_q):= \F_q \cup \{\infty\}$. Consider the map $\gamma(u)= \frac{u-\sqrt{a}}{u+\sqrt{a}}$ for $u \in \mathbb{D}_{q}:= \mathbb{P}^{1}(\F_q) \setminus \{\pm \sqrt{a}\}$. In \cite{QP15} it is proved that $\gamma$ induces an isomorphism between the functional graph $\mathcal{G}(R_n/\mathbb{D}_{q})$ and the functional graph of the map $x\to x^n$ over the multiplicative subgroup of $\F_{q^2}^{*}$ of order $q-\chi_{q}(a)$. This last map is conjugated to the multiplication-by-$n$ map $\Gamma_n : \Z/ \n \rightarrow \Z/{\n}$ where $\n=(q-\chi_{q}(a))\Z$ and we can apply Theorem \ref{th:main} to obtain the following description of the functional graph the R{\'e}dei function $R_n(x,a)$ over the finite field $\F_q$ (note that in this case $\D=\Z$ is a principal domain) according to Theorem 4.3 of \cite{QP15}.

\begin{theorem} Let $n\in \Z^{+}$, $a \in \F_q^{*}$ and $R_n=R_n(x,a)$. Write $q-\chi_q(a)=\nu \cdot \omega$ with $\nu, \omega \in \Z^+$ are such that $\mathrm{rad}(\nu)\mid \mathrm{rad}(n)$ and $\gcd(n, \omega)=1$. Then, $$\mathcal{G}(R_n/\mathbb{D}_{q}) = \bigoplus_{d \mid \omega} \frac{\varphi(d)}{ \ord{n}{d}  } \times \cyc\left( \ord{n}{d}, T_{\nu(n)} \right).$$
where $d$ runs over the positive divisors of $\omega$ and $T_{\nu(n)}$ is the tree associated with the $\nu$-series $\nu(n)$.

\end{theorem}

\subsection{Chebyshev polynomials over finite fields}
The Chebyshev polynomials are defined recursively as follows: $T_0(x)=2$, $T_1(x)=x$ and $T_n(x)=xT_{n-1}(x)-T_{n-2}(x)$ for $n\geq 2$. It is well known that $T_n(x)$ is a monic, degree-$n$ polynomial verifying the functional equation $T_n(x+x^{-1}) = x^n+x^{-n}$. In \cite{QP18} the authors describe the functional graph $\mathcal{G}(T_n/\F_q)$ associated with the Chebyshev polynomial $T_n$ over the finite field $\F_q$. Let $T$ be a tree in $\mathcal{G}(T_n/\F_q)$ attached to a cyclic (periodic) point $c$. We say that $T$ is a {\it generic tree} when $c\neq \pm 2$ (note that $2$ and $-2$ are fixed points of $T_n$). We can apply Theorem \ref{th:main} to describe the generic trees in the functional graph $\mathcal{G}(T_n/\F_q)$. Let $H$ be the multiplicative subgroup of $\F_{q^2}^{*}$ of order $q+1$, $\tilde{F}_{q}= \F_{q} \cup H$, $r_n : \tilde{F}_{q} \rightarrow \tilde{F}_{q}$ be the power map $r_n(\alpha)=\alpha^n$ and $\eta: \tilde{F}_{q} \rightarrow {F}_{q}$ given by $\eta(\alpha)= \alpha + \alpha^{-1}$. The following commutative diagram holds (see for instance \cite{QP18}):

\[ \begin{tikzcd}
\tilde{F}_{q} \arrow{r}{r_n} \arrow[swap]{d}{\eta} & \tilde{F}_{q} \arrow{d}{\eta} \\%
\F_q \arrow{r}{T_n}& \F_q,
\end{tikzcd}
\]

Thus, $\eta$ induces a graph homomorphism from $\mathcal{G}(r_n/\tilde{F}_{q})$ onto $\mathcal{G}(T_n/\F_{q})$. This map preserve periodic points and also the trees attached to the periodic points $\alpha \neq \pm 1$, i.e. the hanging tree of $\alpha$ in $\mathcal{G}(r_n/\tilde{F}_{q})$ is isomorphic to the hanging tree of $\eta(\alpha)$ in $\mathcal{G}(T_n/\F_{q})$; see \cite{QP18}.

Note that $\mathcal{G}(r_n/\F_q)=\mathcal{G}(\Gamma_{n,\ideal{a}})$ and $\mathcal{G}(r_n/H)=\mathcal{G}(\Gamma_{n,\ideal{b}})$ for the ideals $\ideal{a}=(q-1)\Z$ and $\ideal{b}=(q+1)\Z$. By Theorem \ref{th:main}, the trees attached to the periodic points in $\mathcal{G}(r_n/\F_q)$ and $\mathcal{G}(r_n/H)$ are isomorphic to $T_{a_0(n)}$ and $T_{b_0(n)}$, where $q-1=a_0\cdot a_1$ and $q+1=a_1b_1$ are the $n$-decomposition of $q-1$ and $q+1$, respectively. The sets $\F_q$ and $H$ are forward $r_n$-invariant but they are not backward $r_n$-invariant, thus $\mathcal{G}(r_n/\tilde{F}_{q})$ does not split into the disjoint union of $\mathcal{G}(r_n/\F_q)$ and $\mathcal{G}(r_n/H)$. Let $\tilde{S}$ denote the set of vertices in $\mathcal{G}(r_n/\tilde{F}_{q})$ which are in the same connected component of $1$ or $-1$, $\tilde{R}=\F_{q}^{*}\setminus \tilde{S}$ and $\tilde{Q}=H\setminus \tilde{S}$. These sets are both backward and forward $r_n$-invariant and we have a decomposition $$\mathcal{G}(r_n / \tilde{F}_{q}) = \mathcal{G}(r_n/ \tilde{R}) \oplus \mathcal{G}(r_n/ \tilde{Q}) \oplus \mathcal{G}(r_n/ \tilde{S}).$$ From the commutative diagram above we obtain an analogous decomposition for the functional graph of $T_n$: $$\mathcal{G}(r/\mathbb{P}^{1}(\F_q)) = \mathcal{G}(T_n/R) \oplus  \mathcal{G}(T_n/Q) \oplus\mathcal{G}(T_n/S).$$ The component $\mathcal{G}(T_n/S)$ corresponds to the connected components containing the points $\pm 2$. The trees attached to the periodic points in  $\mathcal{G}(T_n/R)$, $\mathcal{G}(r_n/ \tilde{R})$ and $\mathcal{G}(r_n/ \F_q^{*})$ are isomorphic, and the trees attached to the periodic points in $\mathcal{G}(T_n/Q)$, $\mathcal{G}(r_n/ \tilde{Q})$ and $\mathcal{G}(r_n/ H)$ are isomorphic. Let $\chi_q$ be the quadratic character of $\F_{q}$ (i.e. $\chi_q()$). If $\alpha \in \tilde{R} \cup \tilde{Q}$ and $a=\eta(\alpha)=\alpha+\alpha^{-1}$, then $\alpha$ is a root of the equation $X^2-aX+1=0$ with discriminant $a^2-4$. Then, $\alpha \in \tilde{R}$ if $\chi_q(a^2-4)=1$ and $\alpha \in \tilde{H}$ if $\chi_q(a^2-4)=-1$. From the above discussion we have the following description for the generic trees in $\mathcal{G}(T_n/\F_q)$, which is in accordance with Theorem 2 of \cite{QP18}.
\begin{prop}
Let $q-1=a_0\cdot b_0$ and $q+1=a_1b_1$ be the $n$-decomposition of $q-1$ and $q+1$, respectively. Let $a \neq 2$ be a periodic point for the Chebyshev polynomial $T_n: \F_q \rightarrow \F_q$ and $T$ be the corresponding hanging tree of $a$ in $\mathcal{G}(T_n/\F_q)$. Then, $$T= \left\{ \begin{array}{ll}
T_{a_0(n)} & \textrm{ if } \chi_q(a^2-4)=1; \\
T_{a_1(n)} & \textrm{ if } \chi_q(a^2-4)=-1.
\end{array} \right.$$
where $T_{a_0(n)}$ and $T_{a_1(n)}$ are the trees associated with the $\nu$-series $a_0(n)$ and $a_1(n)$, respectively.
\end{prop}

\subsection{Maps induced by endomorphism of ordinary elliptic curves over finite fields}

In \cite{Ugolini18} Ugolini studied the functional graph of maps induced by certain endormorphisms of ordinary elliptic curves whose endomorphism ring $\D$ is isomorphic to the maximal order $\mathcal{O}_{\mathbb{K}}$ of a quadratic imaginary field $\mathbb{K}$ (in particular $\D$ is a Dedekind domain). To simplify let us assume that the characteristic of $\F_q$ is neither $2$ or $3$, and consider an elliptic curve over $\F_q$ of the form $E: Y^2=f(X)$ with $f(X)=X^3+a_1X+a_2$ and $4a_1^2+27a_2^3 \neq 0$. Ugolini considers an endomorphism $\alpha: E \rightarrow E$ of the form $\alpha(x,y)= (\alpha_1(x), y \alpha_2(x))$ with $\alpha_1(x)=a(x)/b(x)$, $a(x),b(x) \in \F_q[x]$ and $\gcd(a(x),b(x))=1$; and the map $r: \mathbb{P}^{1}(\F_{q^n}) \rightarrow \mathbb{P}^{1}(\F_{q^n})$ given by $r(x)=\alpha_1(x)$ if $b(x)\neq 0$ and $r(x)=\infty$ otherwise.

Since $\alpha(x_0,0)=(\alpha_1(x_0),0)$ if $b(x_0)\neq 0$ and $\alpha(x_0,0)=\mathcal{O}$ otherwise, if $x_0\in \F_q$ is a root of $f$  we have that $r(x_0)$ is either a root of $f$ or $\infty$. Thus, the set $Z_f := \{x\in \F_q: f(x)=0\}\cup \{\infty\}$ is forward $r$-invariant. Let $p \in \mathbb{P}^{1}(\F_{q^n})$ be a periodic point of $r$ and $T$ be the corresponding hanging tree of $p$ in the functional graph $\mathcal{G}\left(r/\mathbb{P}^{1}(\F_{q^n})\right)$. We say that $T$ is a {\it generic tree} when $p\not\in Z_f$ (note that $\#\Z_f \leq 4$). We can obtain an explicit description of the generic trees in the functional graph of $r$ in terms of $\nu$-series, as a consequence of Theorem \ref{th:main}. If we denote by $\mathcal{X}=\{P \in E(\overline{\F_q}) : x(P)\in \F_{q^n}\} \cup \{ \mathcal{O} \}$ we have the following commutative diagram:

\[ \begin{tikzcd}
\mathcal{X} \arrow{r}{\alpha} \arrow[swap]{d}{x} & \mathcal{X} \arrow{d}{x} \\%
\mathbb{P}^{1}(\F_{q^n}) \arrow{r}{r}& \mathbb{P}^{1}(\F_{q^n}),
\end{tikzcd}
\]

where $x$ denotes the map taking the $x$-coordinate of the point $P\in \mathcal{X}$ (by convention $x(\mathcal{O})=\infty$).\\

Moreover, Ugolini showed that $\mathcal{X}= E(\F_{q^n}) \cup E(\F_{q^n})_{B_n}$ where $E(\F_{q^n})_{B_n} =\{(x,y) \in E(\F_{q^2}): x \in \F_{q^n}, y \in \F_{q^2}\setminus \F_{q^n}  \}$ and there are morphisms of $\D$-modules: $E(\F_{q^n}) \simeq \D/\n$ and $E(\F_{q^n})_{B_n} \simeq \D/\m$ where  $\D= \mathcal{O}_{\mathbb{K}}$, $\n=\an{\pi_q^n-1}$ and $\m= \an{\pi_q^n +1}$ (being $\pi_q$ the Frobenius endomorphism $\pi_q(x,y)=(x^q,y^q)$). If $a$ denotes the element of $\D$ corresponding to the endomorphism $\alpha$ we have the following isomorphism of functional graphs: $ \mathcal{G}(\alpha/E(\F_{q^n}) ) \simeq \mathcal{G}(\amap{\n} )$ and $\mathcal{G}(\alpha/E(\F_{q^{2n}})_{B_n} ) \simeq \mathcal{G}(\amap{\m} )$.\\

Note that the sets $E(\F_{q^n})$ and $E(\F_{q^n})_{B_n}$ are forward $\alpha$-invariant. If we denote by $\tilde{S}$ the set of vertices in $\mathcal{G}(\alpha/ \mathcal{X})$ which are in the same connected component of some point $P\in \mathcal{X}$ with $x(P)\in \Z_f$ and define $\tilde{R} = E(\F_{q^n}) \setminus \tilde{S}$ and $\tilde{Q}= E(\F_{q^{2n}})_{B_n} \setminus \tilde{S}$, the following decomposition holds: $$\mathcal{G}(\alpha / \mathcal{X}) = \mathcal{G}(\alpha/ \tilde{R}) \oplus \mathcal{G}(\alpha/ \tilde{Q}) \oplus \mathcal{G}(\alpha/ \tilde{S}).$$ From the commutative diagram above, a similar decomposition for the functional graph of $r$ holds: $$\mathcal{G}(r/\mathbb{P}^{1}(\F_{q^n})) = \mathcal{G}(r/R) \oplus  \mathcal{G}(r/Q) \oplus\mathcal{G}(r/S).$$ The map $x$ does not preserve the length of the cycles but it preserves the rooted trees attached to the periodic points $c\not\in Z_f$. Note that if $c \in R$, there is a point $y \in \F_{q^n}^{*}$ such that $f(c)=y^2$ (i.e. $f(c)$ is a non-zero square in $\F_{q^n}$); and if $c \in Q$, there is a point $y \in \F_{q^{2n}}\setminus \F_{q^n}$ such that $f(c)=y^2$ (i.e. $f(c)$ is a nonsquare in $\F_{q^n}$). From the above discussion we obtain the following description for the generic trees in the functional graph $\mathcal{G}(r/\mathbb{P}^{1}(\F_{q^n}))$.

\begin{prop}
Let $E: Y^2= X^3+a_1X+a_2$ be an elliptic curve over $\F_q$, $\alpha: E \rightarrow E$ and $r:\mathbb{P}^{1}(\F_{q^n}) \rightarrow \mathbb{P}^{1}(\F_{q^n}) $ be an endomorphism of $E$ defined over $\F_q$ and its associate rational map, respectively. Suppose that the endomorphism ring of $E$ is isomorphic to a maximal order $\D$ of an imaginary quadratic field and $\alpha=(\alpha_1(x), y \alpha_2(x))$ with $\alpha_1, \alpha_2 \in \F_q(x)$. Let $a$ be the element of $\D$ corresponding to the endomorphism $\alpha$, $\pi_q$ be the Frobenius endomorphism and $\n$ and $\m$ be the ideals of $\D$ generated by $\pi_q^n-1$ and $\pi_q^n+1$, respectively. Consider the $a$-decomposition of the ideals $\n$ and $\m$ given by $\n=\n_0\n_1$ and $\m=\m_0\m_1$. Let $c \in \F_q$ be a periodic point of $r:\mathbb{P}^{1}(\F_{q^n}) \rightarrow \mathbb{P}^{1}(\F_{q^n})$ such that $c^3+a_1c+a_2 \neq 0$ and $T$ its corresponding hanging tree in the functional graph of $r$. Then,
$$T= \left\{ \begin{array}{ll}
T_{\n_0(a)} & \textrm{ if } \chi_{q^n}(c^3+a_1c+a_2)=1; \\
T_{\m_0(a)} & \textrm{ if } \chi_{q^n}(c^3+a_1c+a_2)=-1;
\end{array} \right.$$  
where $T_{\n_0(a)}$ and $T_{\m_0(a)}$ are the trees associated with the $\nu$-series $\n_0(a)$ and $\m_0(a)$, respectively.
\end{prop}

Let's consider an example taken from \cite{Ugolini18} (Example 4.1)

\begin{example}
Let $E:Y^2=X^3-X$ be the elliptic curve defined over $\F_{73}$ and consider the endomorphism $\alpha(x,y)=(\alpha_1(x), y \alpha_2(x))$ with $\alpha_1(x)=\frac{-3(x^{10}-3x^8+5x^6-5x^4+3x^2-1)}{x^9-28x^7-21x^5+28x^3+x}$. Let $r: \mathbb{P}^{1}(\F_{73}) \rightarrow \mathbb{P}^{1}(\F_{73})$ be the map induced by $\alpha$ as above. The endomorphism ring of $E/\F_{73}$ is isomorphic to $\D= \Z[i]$ and the element of $\D$ corresponding to $\alpha$ is $a=3-i$. The Frobenius endomorphism is represented by $\pi_{73} = -3+8i$. In this case, none of the roots of $X^3-X=0$ is a periodic point of $r$ and $\infty$ is a fixed point. The periodic points of $r$ in $\F_{73}$ are $c=52$, $29$, $59$, $30$, $21$, $44$, $14$ and $43$; for any of these values of $c$ we have $\chi_{73}(c^3-c)=-1$. The $a$-decomposition of the ideal $\m= \langle \pi_{73}+1 \rangle$ is given by $\m=\m_0\cdot \m_1$ with $\m_0=\langle 1+i \rangle^2$ and $\m_1= \langle -1+4i \rangle$. Since $\langle a \rangle = \langle 1+i \rangle\cdot \langle 1-2i \rangle$, we have $\m_0(a)=\left( N_{\D}(1+i), N_{\D}(1+i) \right)=(2,2)$. Therefore, all the trees attached to the periodic points $c\in \F_{73}$ of $r$ are isomorphic to $T_{(2,2)}= \langle \langle 2\times \bullet \rangle \rangle$.

\end{example}

\subsection{Linearized polynomials over finite fields}

Here we use Theorem \ref{th:main} to provide a complete description of the dynamic of linearized polynomials over finite fields, extending results of \cite{MV88, PR18}. As mentioned in the introduction, these polynomials, which induces linear maps over finite fields, appear in many practical applications. They also play an important roll in error-correcting-codes in the rank metric \cite{XLYS13}. Let $p$ denote the characteristic of $\F_q$.

\begin{define}
For $f\in \F_q[x]$ with $f(x)=\sum_{i=0}^ma_ix^i$, $L_f(x)=\sum_{i=0}^ma_ix^{q^i}$ is the $q$-associate (or the linearized polynomials) of $f$.
\end{define}

From the Frobenius identity $(a+b)^p=a^p+b^p$ for any $a, b\in \F_{q^n}$, we observe that, for any $f\in \F_q[x]$, the map $c\mapsto L_f(c)$ is an $\F_q$-linear map of $\F_{q^n}$. The dynamics of maps $c\mapsto L_f(c)$ over the finite field $\F_{q^n}$ were previously studied. In~\cite{MV88}, the authors explore the case $\gcd(f(x), x^n-1)=1$, where the linear map induced by $L_f$ is, in fact, a permutation of $\F_{q^n}$. More recently~\cite{PR18}, the authors explore the case that $\gcd(f(x), x^n-1)$ is an irreducible polynomial. It is worth mentioning that a general method to describe the dynamics of linear maps over finite fields is given in~\cite{T05}. However, as pointed out in~\cite{PR18}, such a method cannot readily describe the dynamics of maps $L_f$ over $\F_{q^n}$ if $n$ is divisible by the characteristic $p$ of $\F_q$. We provide a simple description for the functional graph of $L_f$ in terms of the factorial decomposition of $f$.\\

Next we describe the dynamics of the map $c\mapsto L_f(c)$ over $\F_{q^n}$, without any restriction on the polynomial $f\in \F_q[x]$ and the positive integer $n$. We recall that an element $\beta\in \F_{q^n}$ is said to be \emph{normal} over $\F_q$ if the set $\{\beta, \beta^q, \ldots, \beta^{q^{n-1}}\}$ comprises a basis for $\F_{q^n}$ as an $\F_q$-vector space. It is well-known that normal elements exist in any finite field extension. The following lemma provides some nice properties of the $q$-associate of polynomials over $\F_q$. Its proof is direct by calculations so we omit.

\begin{lemma}\label{lem:q-associate}
For $f, g\in \F_q[x]$, we have that $L_{f+g}(x)=L_f(x)+L_g(x)$ and $L_{fg}(x)=L_f(L_g(x))$.
\end{lemma}
We have the following lemma.
\begin{lemma}
Let $\beta\in \F_{q^n}$ be a normal element. Then the map $\Pi_{\beta}:\frac{\F_q[x]}{\an{x^n-1}}\to \F_{q^n}$ given by $g\mapsto L_g(\beta)$ is an isomorphism of $\F_q$-vector spaces.
\end{lemma}
\begin{proof}
The map $\Pi_{\beta}$ is well defined because if $f, g \in \F_q[x]$ are such that $f(x)-g(x)=(x^n-1)h(x)$ for some $h \in F_q[x]$ we have $Lf(x)-Lg(x)=L_h^{q^n}(x)-L_h(x)$ (by Lemma~\ref{lem:q-associate}). Since $L_h(\beta)\in \F_{q^n}$, $L_h^{q^n}(\beta)=L_h(\beta)$ and then $L_f(\beta)=L_g(\beta)$. Lemma~\ref{lem:q-associate} also implies that the  map $\Pi_{\beta}$ is linear. It is direct to verify that $\dim_{\F_q}\frac{\F_q[x]}{\an{x^n-1}}=n=\dim_{\F_q}\F_q^n$, hence it suffices to show that $\Pi_{\beta}$ is one to one. We observe that an element in $\frac{\F_q[x]}{\an{x^n-1}}$ has a representative with degree at most $n-1$. In addition, if a nonzero polynomial $g\in \F_q[x]$ of degree at most $n-1$ is such that $L_g(\beta)=0$, the last equality entails that a nontrivial linear combination of the elements $\beta, \beta^q, \ldots, \beta^{q^{n-1}}$ with
coefficients in $\F_q$ vanishes. In particular, such elements are linearly dependent, a contradiction since $\beta$ is normal over $\F_q$. Therefore, $\Pi_{\beta}$ is one to one, hence it is an isomorphism of $\F_q$-vector spaces.
\end{proof}

From Lemma~\ref{lem:q-associate}, for any $f\in \F_q[x]$ with $f(x)=\sum_{i=0}^na_ix^i$, if $\Gamma_{f}:\frac{\F_q[x]}{\an{x^n-1}}\to \frac{\F_q[x]}{\an{x^n-1}}$ denotes the multiplication-by-$f$-map $g\mapsto f\cdot g$, we have the following commutative diagram:

\[ \begin{tikzcd}
\frac{\F_q[x]}{\an{x^n-1}} \arrow{r}{\Gamma_f} \arrow[swap]{d}{\Pi_{\beta}} & \frac{\F_q[x]}{\an{x^n-1}} \arrow{d}{\Pi_{\beta}} \\%
\F_{q^n} \arrow{r}{L_f}& \F_{q^n}.
\end{tikzcd}
\]

Since $\Pi_{\beta}$ is an isomorphism, we infer the dynamics of the map $c\mapsto L_f(c)$ over $\F_{q^n}$ has the same cycle structure of the map $\Gamma_{f}$ over $\frac{\F_q[x]}{\an{x^n-1}}$. Since $\F_q[x]$ is a residually finite Dedekind domain (actually, is an Euclidean domain), Theorem~\ref{th:main} applies to the map $\Gamma_f$. The arithmetic functions appearing in Theorem~\ref{th:main} are readily computed for $\F_q[x]$. In fact, since $\F_q[x]$ is an Euclidean domain, we can speak of the norm and the Euler Phi function evaluated at elements of $\F_q[x]$ instead of its ideals. For $\D=\F_q[x]$ and a nonzero polynomial $f\in \F_q[x]$, we have that $\N_{\D}(f)=q^{\deg(f)}$. The Euler function is written as $\varphi_{\D}=\Phi_q$ and can be computed as follows: if $g\in \F_q[x]$ is irreducible, $\Phi_q(g)=q^{\deg(g)}-1$ and, if $f\in\F_q[x]$ factors into irreducible polynomials over $\F_q$ as
$$f(x)=f_1(x)^{e_1}\ldots, f_s(x)^{e_s},$$
where $e_i\ge 1$, we have that
$$\Phi_q(f)=q^{\deg(f)}\prod_{i=1}^{s}\left(1-\frac{1}{q^{\deg(f_i)}-1}\right).$$
Additionally, for relatively prime polynomials $f, g\in \F_q[x]$, $\ou(f, g)$ is the least positive integer $k$ such that $f^k\equiv 1\pmod g$. Note also that if $n=p^t \cdot u$ with $p\nmid u$, then $x^n-1 = (x^u-1)^{p^t}$ where $x^u-1$ is a product of irreducible polynomials over $\F_q$. All in all, Theorem~\ref{th:main} entails the following result. 

\begin{theorem}
Let $f\in \F_q[x]$ be a nonzero polynomial and $n=p^t\cdot u$ with $p\nmid u$. Set $h(x):=\gcd(f(x), x^u-1)$ and $S_f(x):=\frac{x^u-1}{h(x)}$. Then the functional graph $\G(L_f/\F_{q^n})$ of the map $c\mapsto L_f(c)$ over $\F_{q^n}$ is given as follows

\begin{equation}\label{eq:linear-maps}
\G(L_f/\F_{q^n})=\bigoplus_{g|S_f^{p^t}}\frac{\Phi_q(g)}{\ou(f, g)}\times \cyc\left(\ou(f, g), T_{h^{p^t}(f)}\right),
\end{equation}
where $g \in \F_q[x]$ runs over the monic divisors of $S_f^{p^t}$ (over $\F_q$) and $T_{h^{p^t}(f)}$ is the tree of the $\nu$-series associated with $h^{p^t}$ and $f$.
\end{theorem}


\section*{Acknowledgments}

The first author was supported by FAPESP under grant 2015/26420-1 and the second author was supported by FAPESP under grant 2018/03038-2.




\begin{thebibliography}{99}

\bibitem{Gassert14}
T.~A.~Gassert. 
\newblock Chebyshev action on finite fields. 
\newblock {\em Discr. Math.} 315: 83--94 (2014). 

\bibitem{Gassert14b}
T.~A.~Gassert. 
\newblock Discriminants of Chebyshev radical extensions. 
\newblock {\em J. Th{\'e}or. Nombres Bordeaux} 26.3: 607--634 (2014).

\bibitem{LN97}
R. Lidl and H. Niederreiter. 
\newblock Finite fields. 
\newblock {\em Cambridge university press} (1997).

\bibitem{MPQ19}
R.~Martins, D.~Panario and C.~Qureshi
\newblock A Survey on Iterations of Mappings over Finite Fields.
\newblock In: {\em Combinatorics and finite fields: Difference sets, 
polynomials, pseudorandomness and applications.} Radon Series on Computational and Applied Mathematics, De 
Gruyter, Berlin, to appear.

\bibitem{M14} C.~Miguel.
\newblock Menon's identity in residually finite Dedekind Domains.
\newblock {\em J. Num. Theory.} 137:179--185 (2014).

\bibitem{MV88} G.L. Mullen and T.P. Vaughan. 
\newblock Cycles of linear permutations over a finite field. 
\newblock {\em Linear Algebra Appl.} 108: 63-82 (1988).

\bibitem{Narkiewicz04} W. Narkiewicz
\newblock Elementary and Analytic Theory of Algebraic Numbers (third edition). \newblock {\em Springer Monogr. Math., Springer-Verlag, Berlin} (2004)

\bibitem{PR18} D.~Panario and L.~Reis.
\newblock The functional graph of linear maps over finite fields and applications.
\newblock{\em Des. Codes Cryptogr.} (2018). \url{https://doi.org/10.1007/s10623-018-0547-5}


\bibitem{PMMY01}
A.~Peinado, F.~Montoya, J.~Munoz and A.~J.~Yuste
\newblock Maximal periods of $x^2+c$ in $\F_q$. 
\newblock {\em In: International Symposium on Applied Algebra, Algebraic Algorithms and Error-Correcting Codes, Springer} pp.~219--228 (2001).

\bibitem{QP15}
C. Qureshi and D. Panario. 
\newblock R\'edei actions on finite fields and multiplication map in cyclic groups. 
\newblock {\em SIAM J. on Discr. Math.} 29: 1486--1503 (2015).

\bibitem{QP18}
C.~Qureshi and D.~Panario.
\newblock The graph structure of Chebyshev polynomials over finite fields and applications.
\newblock {\em Des. Codes Cryptogr.} (2018). \url{https://doi.org/10.1007/s10623-018-0545-7}

\bibitem{QPM17} 
C. Qureshi, D. Panario and R. Martins. 
\newblock Cycle structure of iterating R\'edei functions. 
\newblock {\em Adv. Math. Comm.} 11(2): 397--407 (2017).


\bibitem{Rogers96}
T.~Rogers. 
\newblock The graph of the square mapping on the prime fields. 
\newblock {\em Discr. Math.} 144: 317--324 (1996).

\bibitem{T05} R.~A.~H.~Toledo, 
\newblock  Linear Finite Dynamical Systems. 
\newblock {\em Commun. Algebra.}  33 (9): 2977-2989 (2005).

\bibitem{Ugolini13}
S.~Ugolini
\newblock Graphs associated with the map $x \mapsto x+x^{-1}$ in finite fields of characteristic three and five. 
\newblock{ \em J. Num. Theory} 133: 1207--1228 (2013).

\bibitem{Ugolini14}
S.~Ugolini
\newblock On the iterations of certain maps $x \mapsto k\cdot (x+x^{-1})$ over finite fields of odd characteristic. 
\newblock {\em J. Num. Theory} 142: 274--297 (2014).

\bibitem{Ugolini18}
S.~Ugolini. 
\newblock Functional graphs of rational maps induced by endomorphisms of ordinary elliptic curves over finite fields.
\newblock {\em Periodica Math. Hungarica} 77.2: 237--260 (2018).


\bibitem{VS04}
T. Vasiga and J. Shallit.
\newblock On the iteration of certain quadratic maps over $GF(p)$. 
\newblock {\em Discr. Math.} 277: 219--240 (2004).

\bibitem{WAS13}
A. Wachter-Zeh, V. Afanassiev and V. Sidorenko. 
\newblock Fast decoding of Gabidulin codes. 
\newblock {\em Des. Codes Cryptogr.} 66: 57--73 (2013).

\bibitem{XLYS13}
H. Xie, J. Lin, Z. Yan and B. W. Suter.
\newblock Linearized polynomial interpolation and its applications
\newblock {\em IEEE Trans. on Signal Processing.} 61.6: 206--217 (2013).













%
%

%










\end{thebibliography}
\end{document}